\documentclass[12pt]{amsart}

\usepackage{amsmath,amssymb,amsfonts,amscd,mathtools}
\usepackage{enumitem}
\usepackage{microtype}
\usepackage[colorlinks=true,linkcolor=blue,citecolor=blue,urlcolor=blue]{hyperref}

\numberwithin{equation}{section}

\newtheorem{theorem}{Theorem}[section]
\newtheorem{proposition}[theorem]{Proposition}
\newtheorem{lemma}[theorem]{Lemma}
\theoremstyle{definition}

\theoremstyle{remark}

\newcommand{\Supp}{\operatorname{Supp}}
\newcommand{\codim}{\operatorname{codim}}

\newcommand{\R}{\mathbb{R}}
\newcommand{\Q}{\mathbb{Q}}

\title[On a conjecture of Demailly-Peternell-Schneider]{On a conjecture of Demailly-Peternell-Schneider: The K\"ahler Case}
\author[{Xin Fu, Juanyong Wang}]{Xin Fu and Juanyong Wang} 
\thanks{Xin Fu is  supported by National Key R\&D Program of China 2024YFA1014800 and NSFC No. 12401073. Juanyong Wang is supported by NSFC (Grants No.12301060 and 12288201) and National Key R\&D Program of China (Grant No. 2021YFA1003100).}
\address{School of Science, Institute for Theoretical Sciences, Westlake University, Hangzhou 310030, China}

\email{fuxin54@westlake.edu.cn}
\address{State Key Laboratory of Mathematical Sciences, Academy of Mathematics and Systems Science, Chinese Academy of Sciences, Beijing 100190, China}

\email{juanyong.wang@amss.ac.cn}

\begin{document}

\begin{abstract}
Let $f: (X,\Delta)\to Y$ be a surjective holomorphic map between two normal K\"ahler varieties, where $(X,\Delta)$ is a log canonical pair,  $-(K_X+\Delta)$ is nef  and $Y$ is $\Q$-Gorenstein. In Fu-Guo-Song-Wang \cite{FGSW}, it is proved  that $-K_Y$ is pseudo-effective if $f$ is a projective morphism. In this paper, prove that $K_Y$ is pseudo-effective without assuming that $f$ is projective.

\end{abstract}

\maketitle

\section{Introduction}

Complex Projective (K\"ahler) varieties with numerically effective anticanonical bundles appear naturally in the Minimal Model Program. These objects include Fano varieties, Calabi-Yau varieties, Abelian varieties and so on. A systematic studies of these varieties is initiated by a sequence of foundational works of Demailly-Peternell-Schneider \cite{DPSI,DPSII,DPSIII}. After that, they are many important progress in understanding the structure of singular varieties with anti-nef canonical bundle, c.f. \cite{Cam2,CCM,CaoI,CH1,CH2,MWJ,MW,Peternell,MWWZ,Zha96,Zha19}. 

A key question in understanding the geometry of these varieties is determining fiberation between them. Let $f:X\to Y$ be a surjective morphism with $-K_X$ nef. What can one say about $-K_Y$. In
this direction, it is conjectured by Demailly-Peternell-Schneider \cite{DPSI} and recently by Peternell \cite{Peternell} that $-K_Y$ is pseudo-effective.  Recent progress in singularity theory extends the question to singular pairs with anti-nef canonical bundle. Indeed  the Demailly-Peternell-Schneider conjecture is firstly proved by Chen-Zhang \cite{CZ}, in the case $(X,\Delta)$ is a anti-nef log canonical pair and $Y$ is a $\mathbb Q$-Gorenstein projective variety, using weak positivity theory originally developed by Viehweg \cite{V}, and later on, generalized by Campana \cite{Cam2} and Lu \cite{Lu}. The recent work of Fu-Guo-Song-Wang \cite{FGSW} extends Chen-Zhang's Theorem to a projective morphism between K\"ahler varieties admit the same type of singularities, using the tool of variation of Bergman metrics \cite{BP,PT,DWZZ}. See also Fu-Han-Zou \cite{FHZ} for partial progress on this conjecture   in the K\"ahler case.

The purpose of this paper is to prove the following result in the compact K\"ahler setting without projectivity assumptions, allowing singularities on both the total space and the base.

\begin{theorem}\label{thm:main}
Let $f: (X,\Delta)\longrightarrow Y$ be a surjective holomorphic map satisfying
\begin{itemize}
\item $X$ and $Y$ are normal compact K\"ahler spaces,
\item $Y$ is $\Q$-Gorenstein,
\item $\Delta\geq 0$ is an $\R$-divisor, $(X,\Delta)$ is log canonical, and $K_X+\Delta$ is $\R$-Cartier,
\item $-(K_X+\Delta)$ is nef,
\end{itemize}
then $-K_Y$ is pseudo-effective.
\end{theorem}

\subsection{Outline of the proof}
For simplicity, we assume that the base space $Y$ of the fiberation $f$ is smooth. Pick a log resolution $\pi: M\to X$, set $p=f\circ\pi$, and write
\[
 K_M+\Delta_M=\pi^*(K_X+\Delta)+E,
\]
where $E,\Delta_M$ are effective divisors and the coefficients of $\Delta_M$ are no bigger than one. 
By passing to a modification $W$ of $M$, Campana's relative rational quotient is resolved to a diagram
\begin{equation}\label{eq:intro-diagram}
\begin{CD}
 W @>{q}>> M\\
 @V{\rho}VV @VV{p}V\\
 Z @>{h}>> Y,
\end{CD}
\end{equation}
with $W$ and $Z$ being smooth compact K\"ahler manifolds. The general fibres of $q$ is rationally connected and for very general $y\in Y$, the morphism $q_y: W_y\to Z_y$ is a holomorphic model of the MRC quotient of $X_y$.

By Claudon--H\"oring \cite{CH}, the fibrewise MRC morphism $q_y$ is projective.  By Fu-Guo-Song-Wang \cite{FGSW}, the base space $Z_y$ of the projective MRC quotient $q_y$ has zero Kodaira dimension
\[
 \kappa(Z_y)=0
\]
for very general $y$. Then the variation of relative Bergman construction developed by \cite[Theorem 2.3]{Wang}  yields pseudo-effectivity of the resolved relative canonical class $K_{Z/Y}$.

Then we treat the  rationally connected part $W\to Z$. By Fu-Guo-Song-Wang \cite[Proposition 3.2]{FGSW}, $q^*K_Z+F$ is pseudo-effective, where $F$ is exceptional over $X$. Adding $q^*K_{Z/Y}$ yield that 
\[
 -p_W^*K_{Y}+F
\]
is pseudo-effective. The final step is a current-theoretic descent argument by showing that exceptional term $F$ has zero pushforward to $Y$.  Thus $-K_Z$ is pseudo-effective.

The proof is organized as follows. Section~\ref{sec:prelim} fixes the current formalism on normal spaces and records the fibre integration. Section~\ref{sec:mrc} constructs the relative MRC model, proves the generic-restriction statement, and establishes the corrected pseudo-effectivity of $K_{Z/Y}$ where $Z$ has generic fiber with zero Kodaira dimension over $Y$. Section~\ref{sec:proof} finishes the proof.

\medskip

\textbf{Acknowledgement:} The first author is grateful to Professor Campana for his patient and thorough explanation of MRC fiberation during his visit to Beijing. 
\section{Preliminary lemmas}\label{sec:prelim}

\subsection{Forms and currents on normal spaces}

All complex spaces in this paper are irreducible. We refer the reader to \cite[Section 2.1]{MW} for  differential form and current conventions on analytic spaces. 

The most relevant concept for us is the proper direct image convention used below. If $g:X\to Y$ is proper map between complex analytic spaces, $R$ is a current on $X$, and $\varphi$ is a compactly supported smooth test form on $Y$ (Here a smooth form on a reduced analytic space is locally the restriction of a smooth ambient form under a local embedding), then
\[
 \langle g_*R,\varphi\rangle:=\langle R,g^*\varphi\rangle.
\]
The direct image satisfies
\[
 g_*(R\wedge g^*\eta)=g_*R\wedge\eta
\]
for every smooth form $\eta$ on $B$ for which the products have the right bidegrees. If $u$ is a distribution and $\Omega$ is a smooth closed form, the same definition gives
\[
 g_*(dd^c u\wedge\Omega)=dd^c g_*(u\Omega).
\]


We say that a statement holds for \emph{very general} points of an irreducible complex space if it holds outside a countable union of proper closed analytic subsets.

Fix a holomorphic map $V\to X$ with connected fibers, we prepare two lemmas which will be used to push forward pseudo-effective class on $V$ to pseudo-effective class on $X$.
\begin{lemma}\label{lem:descent1}
Let $\mu: V\to X$ be a proper modification from a smooth compact K\"ahler manifold to a normal compact K\"ahler space, and let $\alpha\in H^{1,1}(X,\R)$. Then $\alpha$ is pseudo-effective if and only if
\[
 \mu^*\alpha+c_1(F)
\]
is pseudo-effective for some $\mu$-exceptional $\R$-divisor $F$, not necessarily effective.
\end{lemma}

\begin{proof}
This is \cite[Lemma 3.1]{FGSW}. We emphasize that  $F$ can have both positive and negative coefficients.
\end{proof}

\begin{lemma}\label{lem:descent2}
Let $\mu: V\to X$ be a modification from a compact K\"ahler manifold to a normal compact K\"ahler space, let $f: X\to Y$ be a surjective holomorphic map onto a normal compact K\"ahler space, and put $p=f\circ\mu$. Let $L$ be an $\R$-line bundle on $Y$. Suppose that there is an effective $\R$-divisor $G$ on $V$ and a $\mu$-exceptional $\R$-divisor $F$, with no sign condition on $F$, such that
\[
 p^*c_1(L)+[G]+[F]
\]
is pseudo-effective and
\[
 \codim_Y p(\Supp G)\geq 2.
\]
Then $L$ is pseudo-effective.
\end{lemma}

\begin{proof}
Put $r=\dim X-\dim Y$. Choose a K\"ahler form $\omega_X$ on $X$ in the sense fixed above and set $\Omega=\mu^*\omega_X$. Then $\Omega$ is a smooth closed semipositive $(1,1)$-form on $V$. Choose a smooth curvature form $\theta$ representing $c_1(L)$ and a closed positive current $T$ in the class $p^*c_1(L)$. Since $V$ is a compact K\"ahler manifold, the $dd^c$-lemma for currents gives a global distribution $u$ such that
\begin{equation}\label{eq:current-decomposition}
 T=p^*\theta+[G]+[F]+dd^c u.
\end{equation}
The wedge $T\wedge\Omega^r$ is well-defined because $\Omega$ is smooth, and
\[
 S:=p_*(T\wedge\Omega^r)
\]
is a closed positive $(1,1)$-current on $Y$.

The current $p_*(\Omega^r)$ is a positive $d$-closed $(0,0)$-current, hence a constant $c$ on the irreducible space $Y$. We claim that $c>0$. Let $X^\circ\subset X_{\mathrm{reg}}$ be a dense open subset over which $\mu$ is an isomorphism. Since $X\setminus X^\circ$ is a proper analytic subset, a general fibre $X_y$ is not contained in it. On the nonempty open subset
\[
 \mu^{-1}(X_y\cap X^\circ)\subset V_y
\]
the map $\mu$ is biholomorphic and $\Omega|_{V_y}=\mu^*(\omega_X|_{X_y})$ is strictly positive in the fibre directions. Hence $\Omega^r|_{V_y}$ has strictly positive mass there, and therefore
\begin{equation}\label{eq:fibre-volume}
 c=\int_{V_y}\Omega^r>0
\end{equation}
for every $y$ in the dense smooth, pure-dimensional locus of $p$. By the projection formula, proved above by duality,
\begin{equation}\label{eq:theta-push}
 p_*(p^*\theta\wedge\Omega^r)=\theta\wedge p_*(\Omega^r)=c\theta.
\end{equation}

The current $p_*([G]\wedge\Omega^r)$ is closed and positive and is supported on $p(\Supp G)$. This analytic set has codimension at least two, so the first support theorem \cite{Dem} gives
\begin{equation}\label{eq:G-vanish}
 p_*([G]\wedge\Omega^r)=0.
\end{equation}

Write $F=F^+-F^-$ with $F^\pm\geq0$ and $F^\pm$ $\mu$-exceptional. The positive closed $(1,1)$-current $\mu_*[F^\pm]$ is supported on $\mu(\Supp F^\pm)$, which has codimension at least two; hence it vanishes by the same support theorem. The projection formula then gives
\begin{equation}\label{eq:F-vanish}
 \mu_*([F^\pm]\wedge\Omega^r)
 =\mu_*[F^\pm]\wedge\omega_X^r=0,
\end{equation}
and consequently $p_*([F]\wedge\Omega^r)=0$.

Finally, since $\Omega$ is closed, proper direct image commutes with $dd^c$ and
\[
 p_*(dd^c u\wedge\Omega^r)=dd^c p_*(u\Omega^r).
\]
Pushing forward \eqref{eq:current-decomposition} and using
\eqref{eq:theta-push}--\eqref{eq:F-vanish}, we obtain
\[
 S=c\theta+dd^c p_*(u\Omega^r).
\]
Thus $S/c$ is a closed positive current representing $c_1(L)$. Thus  $L$ is pseudo-effective.
\end{proof}


\section{The relative MRC quotient}\label{sec:mrc}

In this section, we always assume that $f:X\to Y$ is a surjective holomoprhic map with connected fibers. This is unnecessarily but is helpful to ease notations.

In the following lemma, we do a modification of the singular pair $(X,\Delta)$ and then derive some standard properties of  generic fibers.
\begin{lemma}\label{lem:genericfiber}
Let $f: (X,\Delta)\to Y$ satisfy the hypotheses of Theorem~\ref{thm:main}, and let $\pi: M\to X$ be a log resolution such that
\begin{equation}\label{eq:dis}
 K_M+\Delta_M=\pi^*(K_X+\Delta)+E,
\end{equation}
where $\Delta_M,E\geq0$, the two divisors have no common components, $E$ is $\pi$-exceptional, and every coefficients of $\Delta_M$ are no bigger than  one. Then there is a  Zariski open subset $U\subset Y_{\mathrm{reg}}$ such that, for every $y\in U$,
\begin{enumerate}
\item $X_y$ is normal, $M_y$ is smooth, and $\pi_y: M_y\to X_y$ is a log resolution,
\item $\Supp(\Delta_M+E)$ is relatively simple normal crossing over $U$,
\item if $\Delta_y$ denotes the restriction to $X_y$ of the horizontal part of $\Delta$, then one has
\[(K_X+\Delta)|_{X_y}=K_{X_y}+\Delta_y,\]
 \[K_{M_y}+\Delta_{M,y}
 =\pi_y^*(K_{X_y}+\Delta_y)+E_y.
\]
\item $(X_y,\Delta_y)$ is log canonical and $-(K_{X_y}+\Delta_y)$ is nef.
\end{enumerate}
\end{lemma}

\begin{proof}
In the following argument, $U$ denotes a Zariski open set of $Y$, which might change from line to line.

\textbf{Proof of item (1) and (2):} By Frisch's generic flatness theorem \cite{Frisch}, one may assume that the map $f$ is flat by shrinking the Zariski open set $Y_{\mathrm{reg}}$. Since $X$ is normal, its generic fibre is normal. By Fischer's openness of fibrewise normality \cite[Proposition~3.22]{Fischer}, after  further shrinking $U$, one may assume that every fibre $X_y$ is normal.

The maps from the smooth manifold $M$ and from every stratum of the simple normal crossing divisor $\Supp(\Delta_M+E)$ are holomorphic. Thus for a stratum which dominates $Y$, its critical values is a proper analytic subset. Also, the image of a non-dominating stratum is also proper analytic subset. Removing these finitely many proper analytic subvarieties, we now restrict the map $f:M\to Y$ to the preimage of $U$. Then the map $f$ from $f^{-1}(U)$ and all dominating strata to $Y$ are smooth. Thus $M_y$ is smooth and $\Supp(\Delta_M+E)|_{M_y}$ is simple normal crossing for $y\in U$.

\textbf{Proof of item (3) and (4):} Decompose $\Delta=\Delta_{hor}\cup\Delta_{ver}$ to the $f$-horizontal components and $f$-vertical components and let $F$ be the support of the $\pi$ exceptional divisors.
Set\[Z=X_{\mathrm{sing}}\cup\mathrm{\pi}({F})\cup\Delta_{\mathrm{ver}}\]
After removing those components of $Z$, which do not dominate $Y$, then the remaining components of $Z$ have codimension at least two in $X$. Flatness and the fibre-dimension theorem then yield that
\[
 \codim_{X_y}(Z\cap X_y)\geq2
\]
for $y$ in a further shrinked nonempty Zariski open set $U$. The restriction \[\pi_y:M_y\rightarrow X_y\] is proper and is an isomorphism over $X_y\setminus(Z\cap X_y)$, so it is bimeromorphic. Together with the fact that $(M_y,\Supp(\Delta_M+E)|_{M_y})$ is simple normal crossing, $\pi_y$ is a log resolution.

We next verify adjunction at codimension one. Let $x$ be a codimension-one point of $X_y$, then $x\notin Z$. The map $\pi$ is a local biholomorphism near the unique point over $x$, and $M\to Y$ is smooth there; hence $f$ is a submersion near $x$. Since $Y$ is smooth, local coordinates on the base trivialize the normal bundle of the fibre. Ordinary adjunction gives
\[
 K_M|_{M_y}=K_X|_{X_y}=K_{X_y}
\]
at codimension one. Every horizontal component of $\Delta$ meets $X_y$ properly, while no vertical component meets the chosen fibres after the preceding shrinking. Thus the same codimension-one calculation identifies the Weil $\R$-divisor associated with $(K_X+\Delta)|_{X_y}$ with $K_{X_y}+\Delta_y$. The left-hand side is $\R$-Cartier, so $K_{X_y}+\Delta_y$ is $\R$-Cartier and the equality extends over the normal fibre. Item $(3)$ is proved. 

On the smooth fibre $M_y$, ordinary adjunction gives $K_M|_{M_y}=K_{M_y}$. Restricting \eqref{eq:dis} yields \[K_{M_y}+\Delta_{M,y}
 =\pi_y^*(K_{X_y}+\Delta_y)+E_y.
\] Every component of $E_y$ maps into $Z\cap X_y$, which has codimension at least two, so $E_y$ is $\pi_y$-exceptional. Since every coefficient of $\Delta_{M,y}$ is at most one, the restricted discrepancy equality proves that $(X_y,\Delta_y)$ is log canonical.

Finally, it follows from
\[
 -(K_{X_y}+\Delta_y)=-(K_X+\Delta)|_{X_y}
\]
that  $-(K_{X_y}+\Delta_y)$ is nef.
\end{proof}

The following lemma is a preparation of the relative MRC quotient construction in Proposition \ref{prop:relative-mrc}.
\begin{lemma}\label{lem:modifiber}
Let $\mu: X\to Z$ be a proper modification between irreducible reduced complex spaces, and suppose that there are proper surjective holomorphic maps
\[
 g_X: X\to Y,
 \qquad
 g_Z: Z\to Y
\]
to an irreducible complex space such that $g_X=g_Z\circ\mu$. Then there is a Zariski-open subset $U\subset Y$ such that for every $y\in U$, the reduced fibres $X_y$ and $Z_y$ are pure-dimensional, and $\mu_y: X_y\to Z_y$ maps each irreducible component of $X_y$ bimeromorphically onto an irreducible component of $Z_y$. 
\end{lemma}
\begin{proof}
Let $Z^{\circ}\subset Z$ be the largest open subset over which $\mu$ is biholomorphic, and put
\[
 C:=Z\setminus Z^{\circ},
 \qquad
 \widetilde C:=X\setminus\mu^{-1}(Z^{\circ}).
\]
Both $C$ and $\widetilde C$ are nowhere-dense analytic subsets. First remove from $Y$ the images of all irreducible components of $C$ and $\widetilde C$ which do not dominate $Y$. Apply generic flatness to $X$, $Z$, and to the remaining dominating components of these two analytic subsets. The flatness then gives a further nonempty analytic open set on which $X_y$ and $Z_y$ are pure-dimensional of their generic dimensions, whereas $C_y$ and $\widetilde C_y$ have strictly smaller dimension. Thus no irreducible component of $X_y$ or $Z_y$ is contained in the corresponding exceptional subset.

Let $D$ be an irreducible component of $Z_y$. The set $D\cap Z^{\circ}$ is nonempty and dense in $D$, and its inverse image is biholomorphic to it. The closure of this inverse image in $X_y$ is an irreducible component $\widetilde D$ of $X_y$, and the proper map $\widetilde D\to D$ is biholomorphic over the dense open set $D\cap Z^{\circ}$. Hence it is bimeromorphic. Conversely, every component of $X_y$ is obtained in this way because it is not contained in $\widetilde C$. This proves the claim. 
\end{proof}
In the following proposition, based on Campana's fundamental construction of relative MRC quotient, after a modification, we construct a holomorphic model for the generic fiber of relative MRC quotient.
\begin{proposition}\label{prop:relative-mrc}Let $p: M\to Y$ be a surjective holomorphic map with connected fibres from a compact K\"ahler manifold to a normal compact complex space. There exist smooth compact K\"ahler manifolds $W$ and $Z$, a modification $\rho\colon W\to M$, and surjective holomorphic maps
\[
 q: W\to Z,
 \qquad
 h: Z\to Y
\]
such that
\[
 p\circ\rho=h\circ q.
\]
Moreover, the general fibre of $q$ is  rationally connected, and for very general $y\in Y$, the induced map
\[
 q_y: W_y\to Z_y
\]
is a holomorphic model of the MRC quotient of $M_y$ with smooth source and target. 
\end{proposition}
\begin{proof}Campana's relative rational quotient \cite[Theorem~1.1 and Proposition~2.8]{Cam} gives an irreducible compact complex space $R=R(p)$, an almost-holomorphic fibration
\[
 r\colon M\dashrightarrow R,
\]
and a dominant meromorphic map $a\colon R\dashrightarrow Y$ such that $p=a\circ r$. Let $\Gamma_a\subset R\times Y$ be the closure of the graph of $a$, with projections $p_R$ and $p_Y$. For every $y\in Y$, define
\[
 R_y:=\bigl(R(p)\bigr)_y
 :=p_R\bigl(p_Y^{-1}(y)\bigr)\subset R.
\]
 Thus $R_y$ is a compact analytic subset. By \cite[Proposition~2.8]{Cam}, for $y$ very general 
\[
 r_y\colon M_y\dashrightarrow R_y
\]
is the rational quotient of $M_y$. Replacing $R$ by its normalization, one may assume $R$ is normal. Let $R_1$ be the normalization of $\Gamma_a$ and let
\[
 \gamma\colon R_1\to R,
 \qquad
 a_1\colon R_1\to Y
\]
be the projections. Further choose a smooth compact K\"ahler modification
\[
 \nu\colon Z\to R_1
\]
and set
\[
 h:=a_1\circ\nu,
 \qquad
 b:=\gamma\circ\nu\colon Z\to R.
\]
Thus $b$ is a modification. Put
\[
 s:=b^{-1}\circ r\colon M\dashrightarrow Z.
\]

The map $s$ is almost holomorphic. Indeed, choose a nonempty Zariski-open subset $R^0\subset R$ on which $b$ is an isomorphism and which is disjoint from the proper analytic set $r(I_r)$. For $Z^0:=b^{-1}(R^0)$, every fibre of $s$ over $Z^0$ is the corresponding fibre of $r$ and avoids $I_r$. Hence $s(I_s)\cap Z^0=\varnothing$.

We resolve the indeterminacy of $s$ by blow-ups
\[
 \rho: W\to M.
\]
One may also assume that $\rho$ is an isomorphism over the holomorphic locus of $s$. Let $q=s\circ\rho$, then  $q: W\to Z$ is holomorphic and surjective and \[p\circ\rho=h\circ q.\]

For a general $z\in Z$, the point $b(z)$ lies in the locus where $b$ is an isomorphism and $r$ has a smooth regular fibre and $\rho$ is an isomorphism along that fibre. Then $q^{-1}(z)$ is isomorphic to a smooth general fibre of Campana's relative rational quotient. Such a fibre is rationally chain connected and therefore connected. $Z$ is normal, 
so
\[
 q_*\mathcal O_W=\mathcal O_Z,
\]
and every fibre of $q$ is connected. Similarly, for every $y\in Y$,
\[
 h^{-1}(y)=q\bigl((p\circ\rho)^{-1}(y)\bigr)
\]
is also connected.

It remains only to compare the fibres with Campana's fibrewise quotient. After shrinking to a Zariski open subset $U\subset Y_{\rm reg}$, generic smoothness and Lemma~\ref{lem:modifiber}, applied to the normalization $R_1\to\Gamma_a$, to $\nu\colon Z\to R_1$, and to $\rho\colon W\to M$, gives that $M_y$, $W_y$, and $Z_y$ are smooth, $\rho_y\colon W_y\to M_y$ is a modification, and the natural map
\[
 \beta_y\colon Z_y\longrightarrow R_y
\]
induced by $b\colon Z\to R$ is proper and bimeromorphic. 

For very general $y\in U$, by construction,
\[
 \beta_y\circ q_y=r_y\circ\rho_y
\]
as meromorphic maps. Since $\beta_y$ and $\rho_y$ are bimeromorphic, thus $q_y\colon W_y\to Z_y$ is a holomorphic model of the MRC quotient $r_y$. This proves the proposition.
\end{proof}
To apply the variation of Bergman-kernel method, we show that there is a fixed index $m_0$ such that the $m_0$-th pluricanonical bundle of a generic fiber has a section.  
\begin{lemma}\label{lem:index}
Let $h: Z\to Y$ be a proper surjective holomorphic map between compact complex manifolds. Assume that $h$ is smooth over an analytic Zariski-open subset $U\subset Y$ and that
$\kappa(Z_y)=0$ for very general $y\in U$. Then there exists an integer $m_0>0$ and a nonempty analytic Zariski-open subset $U_0\subset U$ such that
\[
 H^0(Z_y,m_0K_{Z_y})\neq 0
\]
for every $y\in U_0$.
\end{lemma}

\begin{proof}
Replacing $U$ by a connected component, one may assume that $U$ is connected. For every $m\geq1$, the line bundle $mK_{Z/U}$ is flat over $U$ by further shrink $U$ if necessary. By Grauert's direct image Theorem \cite{Grauert}
\[
\Sigma_m:=\{y\in U\mid h^0(Z_y,mK_{Z_y})\geq1\}
\]
is a closed analytic subset of $U$.

By assumption, there is a countable union $N$ of proper closed analytic subsets of $U$ such that every $y\in U\setminus N$ satisfies $\kappa(Z_y)=0$. Hence
\[
 U\setminus N\subset \bigcup_{m\geq1}\Sigma_m.
\]
If every $\Sigma_m$ were proper subset of $U$, the connected complex manifold $U$ would be a countable union of proper closed analytic subsets. By the Baire category theorem, the union of $N$ and $\bigcup_{m\geq1}\Sigma_m$ are proper subset of $U$. Thus $U=\Sigma_{m_0}$ for some $m_0$. Taking $U_0=U$ proves the assertion.
\end{proof}

In the following Proposition, we combine the raltive MRC quotient and Lemma \ref{lem:index} to study the canonical bundle of $Z$ over $Y$. The appearance of singularities of $Y$ will make the problem slightly more complicated.
\begin{proposition}\label{prop:ztoy}
Under the hypotheses of Theorem~\ref{thm:main}, let $\pi: M\to X$ be the log resolution in Lemma~\ref{lem:genericfiber}. Let
\[
 W\xrightarrow{q}Z\xrightarrow{h}Y
\]
be the relative MRC model of Proposition~\ref{prop:relative-mrc} applied to $p$. Then there is an effective $\R$-divisor $D$ on $Z$ such that
\begin{enumerate}
\item 
 $K_{Z/Y}+D$
is pseudo-effective,
\item 
 $\codim_Y h(\Supp D)\geq 2.$
\end{enumerate}
\end{proposition}

\begin{proof}
Choose a very general $y$ in the open set supplied by Lemma~\ref{lem:genericfiber} and in the open set of Proposition~\ref{prop:relative-mrc}. The relative MRC quotient gives a meromorphic MRC map
\[
 \psi_y: X_y\dashrightarrow Z_y
\]
whose elimination of indeterminacies is the commutative diagram
\begin{equation}\label{eq:fibre-elimination}
\begin{CD}
 W_y @>{q_y}>> X_y\\
 @V{\mu_y}VV @VV{\psi_y}V\\
 Z_y @>{Id}>> Z_y
\end{CD}
\end{equation}
where
 $\mu_y:=(\pi\circ\rho)|_{W_y}.$
Here $W_y$ and $Z_y$ are smooth compact K\"ahler manifolds and $Z_y$ is not uniruled. Since $q_y$ is a holomorphic model of the MRC quotient with smooth source and target, Claudon--H\"oring's theorem \cite[Theorem~1.2]{CH} implies that $q_y$ is projective. Thus \eqref{eq:fibre-elimination} is precisely the projective elimination required in \cite[Proposition~3.2]{FGSW}. By Lemma~\ref{lem:genericfiber}, the source pair $(X_y,\Delta_y)$ is a normal log canonical compact K\"ahler pair and its anti-log-canonical class is nef. By \cite{FGSW}, one has
\begin{equation}\label{eq:kappa}
 \kappa(Z_y)=0
\end{equation}
for very general $y\in Y$. 

Choose a K\"ahler resolution $\tau: Y'\to Y$. Let $Z'$ be a smooth compact K\"ahler resolution of the main component of $Z\times_Y Y'$, and denote the induced maps by
\begin{equation}\label{eq:basechange}
\begin{CD}
 Z' @>{\sigma}>> Z\\
 @V{h'}VV @VV{h}V\\
 Y' @>{\tau}>> Y.
\end{CD}
\end{equation}
Over the locus where $\tau$ is an isomorphism, the main component agrees with $Z$, hence $\sigma$ is a modification. Moreover,
$h'$ has connected fibres. Its general fibre is smooth and birational to a general fibre of $h$.

By \eqref{eq:kappa} and Lemma~\ref{lem:index}, there is a fixed integer $m_0>0$ such that
\[
 H^0(Z'_{y'},m_0K_{Z'_{y'}})\neq0
\]
for general $y'\in Y'$. Cao's relative Bergman theorem in the form \cite[Theorem~2.3]{Wang}  (see also \cite{Cao}) implies that 
$K_{Z'/Y'}$ is pseudo-effective.

Since $Y$ is $\Q$-Gorenstein, write
\[K_{Y'}+F^-=\tau^*K_Y+F^+,\]
where $F^+$ and $F^-$ are $\tau$-effective exceptional divisors with no common components. Since $Z$ is smooth, we also write
\[
 K_{Z'}=\sigma^*K_Z+B,
\]
where $B$ is $\sigma$-exceptional and effective. Therefore, we have
\begin{equation}\label{eq:AAA}
 \sigma^*K_{Z/Y}+B+h'^*F^-
 =K_{Z'/Y'}+h'^*F^+.
\end{equation}
Set
\[
 D:=\sigma_*(h'^*F^-)\geq0,
\]
then $D$ is a $\R$-Cartier divisor. At the generic point of every prime divisor on $Z$, the coefficient of $D$ is the coefficient of $h'^*F^-$ along its strict transform. Hence
\[
 h'^*F^- -\sigma^*D
\]
is $\sigma$-exceptional with possibly negative coefficients. It follows from $K_{Z'/Y'}$ is pseudo-effective and \eqref{eq:AAA} that
\[
 \sigma^*(K_{Z/Y}+D)+[B+h'^*F^--\sigma^*D]
\]
is pseudo-effective. Since \[B+h'^*F^--\sigma^*D\] is $\sigma$ exceptional, Lemma~\ref{lem:descent1} implies that $K_{Z/Y}+D$ is pseudo-effective. This proves item $(1)$.
If a prime divisor $P\subset Z$ contained in $D$, it is dominated by a component of $h'^*F^-$. The commutative square \eqref{eq:basechange} implies that
\[
 h(P)\subset \tau(\Supp F^-),
\]
and hence
\[
 h(\Supp D)\subset \tau(\Supp F^-).
\]
Every component of $F^-$ is $\tau$-exceptional, and $Y$ is normal, so the right-hand side has codimension at least two. This proves item $(2)$. 
\end{proof}

\section{Proof of the main theorem}\label{sec:proof}
We first prove our main theorem \ref{thm:main} with the additional assumption that the fibers of $f:X\to Y$ is connected. \begin{proof}
Suppose that we are under the assumption of Theorem~\ref{thm:main} and moreover $f$ has connected fibers.
We first pick a log resolution $\pi: M\to X$ supplied by Lemma~\ref{lem:genericfiber}. Then after modification $\rho: W\to M$, Proposition~\ref{prop:relative-mrc} gives a relative MRC factorization  $W\xrightarrow{q}Z\xrightarrow{h}Y$. Next, Proposition~\ref{prop:ztoy} yields a divisor $D\geq0$ on $Z$ such that $K_{Z/Y}+D$ is pseudo-effective.

Put
\[g=\pi\circ\rho:W\to X,\qquad
 p=f\circ\pi: M\to Y,
 \qquad
 p_W=p\circ\rho=h\circ q: W\to Y.
\]
Let
\[
 \alpha:=-g^*c_1(K_X+\Delta)\in H^{1,1}(W,\R).
\]
By assumption, $\alpha$ is nef.


Write
\begin{equation}\label{eqn:adjWY}
 K_W+H
 =g^*(K_X+\Delta)+F, \qquad F,H\geq0,
\end{equation}
with no common components and all the coefficients of $H$ are no larger than $1$. Pick an relative ample $\mathbb Q$-divisor $A$, and fix a  small rational number $\epsilon>0$, one has
\begin{equation}\label{eq:asymklt}
 -K_W+F-\epsilon^2H+\epsilon A
 =-g^*(K_X+\Delta)+(1-\epsilon^2)H+\epsilon A.
\end{equation}
And since \[-g^*(K_X+\Delta)+\epsilon A\]
is $g$-relatively ample, the RHS of \eqref{eq:asymklt} is asymptotically klt, so is the LHS \[N:=-K_W+F-\epsilon^2H+\epsilon A.\] Moreover, $F$ is effective and $\epsilon A+\epsilon^2 H$ is $g$-relatively ample.  Thus for a fixed $\epsilon$ and any sufficiently large and divisible integer $m_0$, one has 
\[f_*(m_0(K_W+N))=f_*(m_0(F-\epsilon^2H+\epsilon A))\]
is a non-zero sheaf. Then applying \cite[Proposition 3.1]{FGSW} yield that \[K_{W/Z}+N=-q^*K_Z+F-\epsilon^2H+\epsilon A\] 
is pseudo-effective. Let $\epsilon\to 0$, one has
\begin{equation}\label{eq:KZZ}-q^*K_Z+F\end{equation} is pseudo-effective.

By Proposition~\ref{prop:ztoy}, $K_{Z/Y}+D$ is pseudo-effective. Since $Z$ is smooth, $D$ is $\R$-Cartier and $q^*D$ is an effective $\R$-divisor. Pulling the class back by $q$ and adding it to \eqref{eq:KZZ}, one obtain
\begin{equation}\label{eq:WY-psef}
 -p_W^*c_1(K_{Y})+F+[q^*D]
 \quad\text{is pseudo-effective.}
\end{equation}

By \eqref{eqn:adjWY}, $F$ is   $g=\pi\circ\rho: W\to X$ exceptional and the map $p_W:W\to Y$ factor through $g$ as $p_W=f\circ g$. Moreover, $p_W(\Supp q^*D)\subset h(\Supp D),$ which has codimension at least two by Proposition \eqref{prop:ztoy}. Applying Lemma~\ref{lem:descent2} to \eqref{eq:WY-psef}, with
\[
 L=-K_Y,
 \qquad
 G=q^*D,
 \qquad
 F=F,
\]
shows that $-K_Y$ is pseudo-effective. This proves Theorem~\ref{thm:main} under the assumption that $f$ has connected fibers.
\end{proof}

Now we proceed to prove the main theorem \ref{thm:main} by removing the assumption that $f$ has connected fibers, using a standard Stein-Factorization argument.
\begin{proof} Suppose that we are under the assumption of Theorem~\ref{thm:main}. Take a Stein factorization of $X\to Y$:
\[X\xrightarrow{\gamma}S\xrightarrow {\nu}Y.\]
Note that $S$ is normal but not necessarily $\mathbb Q$-Gorenstein. Then by the relative MRC quotient construction Proposition \ref{prop:relative-mrc} and Lemma \ref{lem:index}, there are smooth  manifolds $W$ and $Z$, and maps 
\[q:W\to Z, \qquad h:Z\to S\]such that \begin{equation}\label{eqn:Z}-q^*K_{Z}+F\end{equation} is pseudo-effective for some effective divisor $F$ on $W$, and for some Zariski open set $S^\circ$ of $S$, \[H^0(Z_s,m_0K_{Z_s})\neq 0, \,\,\,\textnormal{for}\,\,\,s\in S^\circ\] After choosing a suitable Zariksi open set $Y^\circ$ of $Y$, one may assume that  $\nu^{-1}(Y^\circ)$ is contained in $S^\circ$ and $\nu$ is a \'etale map when restricted to $\nu^{-1}(Y^\circ)$. Thus for $y\in Y^\circ$, one has
\begin{equation}\label{eqn:decom}H^0(Z_y,m_0K_{Z_y})=\bigoplus_{s\in\nu^{-1}(y)} H^0(Z_s,m_0K_{Z_s}).\end{equation}
Thus the argument of Proposition \ref{prop:ztoy} yield a divisor $D$ on $Z$ satisfying \begin{equation}\label{eqn:ztoy2}K_{Z/Y}+D\end{equation} is pseudo-effective,
and $\codim_Y \nu\circ h(\Supp D)\geq 2.$ The point is that the relative Bergman theorem \cite[Theorem~2.3]{Wang} does not require the fiber $Z_y$ to be connected and \eqref{eqn:decom} suffices for the purpose. Then the theorem follows form \eqref{eqn:Z},\eqref{eqn:ztoy2} and Lemma \ref{lem:descent2}.
\end{proof}


\small
\begin{thebibliography}{99}

\bibitem{BS}
C. B\u{a}nic\u{a} and O. St\u{a}n\u{a}\c{s}il\u{a},
\emph{Algebraic methods in the global theory of complex spaces},
Editura Academiei and John Wiley \& Sons, Bucharest--London--New York, 1976.
\bibitem{BP}{B. Berndtsson, M. Paun, \emph {Bergman kernels and the pseudoeffectivity of relative canonical bundles
Berndtsson}, 
Duke Math. J. 145 (2008), no. 2, 341–378}
\bibitem{Cam2}F. Campana \emph{Orbifolds, Special Varieties and Classification Theory}, Ann. Inst.
Fourier (Grenoble) 54 (2004), 499–630
\bibitem{Cam}
F. Campana,
\emph{Orbifolds, special varieties and classification theory: an appendix},
Ann. Inst. Fourier (Grenoble) \textbf{54} (2004), no.~3, 631--665.

\bibitem{Cao}
J. Cao,
\emph{Ohsawa--Takegoshi extension theorem for compact K\"ahler manifolds and applications},
in Complex and Symplectic Geometry, Springer INdAM Ser. \textbf{21}, Springer, Cham, 2017, 19--38.



\bibitem{CCM} F. Campana, J. Cao, and S. Matsumura, \emph{Projective klt pairs with nef anti-canonical divisor},
Algebr. Geom. 8 (2021), no. 4, 430–464.
\bibitem{CaoI} J. Cao \emph{Albanese maps of projective manifolds with nef anticanonical bundles}, 
Ann. Sci. Éc. Norm. Supér. (4) 52 (2019), no. 5, 1137–1154
\bibitem{CH1} J. Cao, and  H\"oring, A. \emph{
Manifolds with nef anticanonical bundle}, J. Reine Angew. Math. 724 (2017), 203–244.
\bibitem{CH2} J. Cao, and  H\"oring, A. \emph{
A decomposition theorem for projective manifolds with nef anticanonical bundle}, J. Algebraic Geom. 28 (2019), no. 3, 567–597.
\bibitem{CH}
B. Claudon and A. H\"oring,
\emph{Projectivity criteria for K\"ahler morphisms},
arXiv:2404.13927.
\bibitem{CZ}
M. Chen and Q. Zhang,
\emph{On a question of Demailly--Peternell--Schneider},
J. Eur. Math. Soc. \textbf{15} (2013), no.~5, 1853--1858.
\bibitem{Dem}
J.-P. Demailly,
\emph{Complex analytic and differential geometry},
Online textbook, version of June 21, 2012,
\url{https://www-fourier.univ-grenoble-alpes.fr/~demailly/manuscripts/agbook.pdf}.
\bibitem{DPSII}{J.P. Demailly, T. Peternell,  and M.Schneider, {\emph K\"ahler manifolds with numerically effective Ricci class}, Comp. Math., 89, 1993,  p. 217--240}
\bibitem{DPSI}
J.-P. Demailly, T. Peternell, and M. Schneider,
\emph{Compact complex manifolds with numerically effective tangent bundles},
J. Algebraic Geom. \textbf{3} (1994), no.~2, 295--345.
\bibitem{DPSIII} J. Demailly, T. Peternell, M. Schneider, \emph{ Pseudo-effective line bundles on
compact K\"ahler manifolds}, Internat. J. Math 12 (2001), 689–741.
\bibitem{DWZZ}F. Deng, Z. Wang, L. Zhang, and X. Zhou \emph{New characterization of plurisubharmonic functions and positivity of direct image sheaves}, Amer. J. Math. 146 (2024), no. 3, 751–768.
\bibitem{PT} {M. Paun and S. Takayama, \emph{Positivity of twisted relative pluricanonical bundles and their direct images}, J. Algebraic Geom. 27 (2018), no. 2, 211–272.}
\bibitem{FGSW}
X. Fu, B. Guo, J. Song, and J. Wang,
\emph{Fundamental groups of compact K\"ahler varieties with nef anti-canonical bundle},
arXiv:2602.07420.

\bibitem{FHZ}
X. Fu, J. Han, and Y. Zou,
\emph{Variation of K\"ahler--Einstein metrics with mixed singularities},
arXiv:2508.02212.

\bibitem{Fischer}
G. Fischer,
\emph{Complex analytic geometry},
Lecture Notes in Math. \textbf{538}, Springer, Berlin--New York, 1976.

\bibitem{Frisch}
J. Frisch,
\emph{Points de platitude d'un morphisme d'espaces analytiques complexes},
Invent. Math. \textbf{4} (1967/68), 118--138.
\bibitem{Grauert} H. Grauert,\emph{Ein Theorem der analytischen Garbentheorie und die Modulr\"aume komplexer Strukturen},
Inst. Hautes Études Sci. Publ. Math. No. 5 (1960), 64 pp.
 \bibitem{Lu} S. Lu, \emph{A refined Kodaira dimension and its canonical fibration}, Preprint. math.AG/0211029
\bibitem{MWJ}S.-I. Matsumura, and J. Wang, \emph{Structure theorem for projective klt pairs with
nef anti-canonical divisor}, Journal of the European Mathematical Society. (online first)
doi.org/10.4171/jems/1702, 202
\bibitem{MW}
S.-I. Matsumura and X. Wu,
\emph{Compact K\"ahler three-folds with nef anti-canonical bundle},
Math. Ann. \textbf{391} (2025), 1253--1289.
\bibitem{MWWZ} S.-I. Matsumura, J. Wang, X. Wu, and  Q. Zhang \emph{Compact K\"ahler manifolds with nef anti-canonical bundle,} arXiv:2506.23218
\bibitem{Peternell} T. Peternell, \emph{Varieties with generically nef tangent bundles}, J. Eur. Math. Soc.
14 (2012), 571–603. 
\bibitem{Wang}
J. Wang,
\emph{On the Iitaka conjecture $C_{n,m}$ for K\"ahler fibre spaces},
Ann. Fac. Sci. Toulouse Math. (6) \textbf{30} (2021), no.~4, 813--897.

\bibitem{V} E. Viehweg, \emph{Weak positivity and the additivity of Kodaira dimension for certain
fiber spaces}, Adv. Stud. Pure. Math. 1 (1983), 329–353.
\bibitem{Zha96} Q. Zhang, \emph{On projective manifolds with nef anticanonical bundles}, J. Reine Angew. Math. \textbf{478} (1996), 57–60.
\bibitem{Zha19} Q. Zhang, \emph{Algebraic fiber spaces with nef anticanonical bundles}, Math. Ann. 332, 697–703 (2005).
\end{thebibliography}
\end{document}